\documentclass[11pt]{article}
\usepackage[letterpaper,margin=1in]{geometry}
\usepackage[T1]{fontenc}
\usepackage[utf8]{inputenc}
\usepackage{amsmath,amssymb,amsthm}
\usepackage{newtxtext,newtxmath}
\usepackage{microtype}
\usepackage{aliascnt}
\usepackage{hyperref}
\usepackage[nameinlink,capitalize,noabbrev]{cleveref}
\hypersetup{colorlinks=true,linkcolor=black,citecolor=black,urlcolor=black,
 pdftitle={},
 pdfsubject={},
 pdfauthor={}}

\newtheorem{theorem}{Theorem}[section]
\newaliascnt{lemma}{theorem}
\newtheorem{lemma}[lemma]{Lemma}
\aliascntresetthe{lemma}
\newaliascnt{corollary}{theorem}
\newtheorem{corollary}[corollary]{Corollary}
\aliascntresetthe{corollary}
\crefname{theorem}{Theorem}{Theorems}
\crefname{lemma}{Lemma}{Lemmas}
\crefname{corollary}{Corollary}{Corollaries}
\newcommand{\logp}{\log^{+}}
\newcommand{\I}{{\mathcal I}}
\newcommand{\J}{{\mathcal J}}
\newcommand{\E}{{\mathbb E}}

\numberwithin{equation}{section}
\title{Beyond Halfway to Hadwiger’s Conjecture}
\author{Chun-Hung Liu\thanks{chliu@tamu.edu. Department of Mathematics, Texas A\&M University, USA. Partially supported by NSF under CAREER award DMS-2144042.}
\and 
Jason Luo\thanks{luojason@mit.edu. Massachusetts Institute of Technology, USA.}}
\date{}

\begin{document}
\maketitle
\vspace{-1.7em}

\begin{abstract}
Hadwiger conjectured in 1943 that every graph with no $K_t$ minor has chromatic number at most $t-1$.
Delcourt and Postle proved that every graph with no $K_t$ minor has chromatic number $O(t\log\log t)$. 
We build on their result to improve this bound to $O(t\log\log\log t)$. 
\end{abstract}

\section{Introduction}\label{sec:introduction}

The chromatic number of a graph\footnote{All graphs are finite, non-null, and simple in this paper, unless otherwise specified.} $G$, denoted by $\chi(G)$, is the minimum $k$ such that $V(G)$ can be covered by $k$ stable sets in $G$, where a {\it stable set} in $G$ is a set of pairwise non-adjacent vertices.
Bounding the chromatic number is a central topic in graph theory.
For example, the Four Color Theorem \cite{AH77,AHK77,IKMMTT,RSST97} states that every planar graph has chromatic number at most $4$.

Hadwiger \cite{Had43} proposed a far-reaching generalization of the Four Color Theorem, stating that for every positive integer $t$, every $K_{t+1}$-minor-free graph has chromatic number at most $t$.
A graph $H$ is a {\it minor} of a graph $G$ if some graph isomorphic to $H$ can be obtained from a subgraph of $G$ by contracting edges.
Wagner proved that the $t=4$ case is equivalent to the Four Color Theorem \cite{W37}.
Robertson, Seymour, and Thomas \cite{RST93} further proved the $t=5$ case.
Hadwiger's conjecture remains open for $t \geq 6$.

Kostochka~\cite{Kos82,Kos84} and Thomason~\cite{Tho84} independently proved that every graph with no $K_t$-minor is $O(t\sqrt{\log t})$-degenerate and hence has chromatic number $O(t\sqrt{\log t})$.
The degeneracy bound proved by Kostochka and Thomason is optimal \cite{Kos82,Kos84,F83}, and the resulting bound on the chromatic number remained the best known for decades.
Breaking the barrier for this degeneracy bound, Norin, Postle, and Song~\cite{NPS23} (essentially) halved the exponent of the logarithm term\footnote{\cite{NPS23} is a combination of two earlier drafts on arXiv \cite{NS19,P19}. The degeneracy barrier was first broken by Norin and Song in \cite{NS19}, and Postle improved the result in \cite{P19}. The title of this paper was inspired by the title of \cite{P19}.}, proving that every $K_t$-minor-free graph has chromatic number $O(t(\log t)^\beta)$ for every fixed $\beta>1/4$. 
Delcourt and Postle~\cite{DP25} further improved the bound to $O(t\log\log t)$. 
The main result of this paper further improves this bound to $O(t\log\log\log t)$.

\begin{theorem}\label{thm:global}
There is an absolute constant $C$ such that, for every integer $t\geq 16$ and every $K_t$-minor-free graph $G$,
\[
\chi(G)\leq Ct\log\log\log t.
\]
\end{theorem}

\cref{thm:global} can be strengthened to odd $K_t$-minor-free graphs.
By {\it contracting an edge-cut} in a graph $G$, we mean that we first find a partition $\{A,B\}$ of $V(G)$ and then contract all edges with one end in $A$ and one end in $B$.
We say that a graph $H$ is an {\it odd minor} of $G$ if some graph isomorphic to $H$ can be obtained from a subgraph of $G$ by repeatedly contracting edge-cuts.
Clearly, every odd minor of $G$ is a minor of $G$.
Gerards and Seymour (see \cite{JT95}) conjectured that every odd $K_{t+1}$-minor-free graph has chromatic number at most $t$, which is a strengthening of Hadwiger's conjecture.
The $t=2$ case is simple since bipartite graphs are exactly odd $K_3$-minor-free graphs.
The $t=3$ case of this conjecture was proved by Catlin \cite{C79}, and Guenin announced a proof for the $t=4$ case (see \cite{S16}).
But this conjecture was disproved for sufficiently large $t$ by K\"{u}hn, Sauermann, Steiner, and Wigderson \cite{KSSW25}; they showed that odd $K_t$-minor-free graphs can have chromatic number at least $(1.5-o(1))t$.
On the other hand, Steiner \cite[Theorem 2]{S22} proved that if every $K_t$-minor-free graph has chromatic number at most $f(t)$, then every odd $K_t$-minor-free graph has chromatic number at most $2f(t)$.
Therefore, \cref{thm:global} immediately implies the following corollary.

\begin{corollary}
There is an absolute constant $C$ such that, for every integer $t\geq 16$ and every odd $K_t$-minor-free graph $G$,
\[
\chi(G)\leq Ct\log\log\log t.
\]
\end{corollary}

\section{Proof overview}

We first introduce notation that will be used in this paper.
Let $G$ be a graph.
Then $\alpha(G)$ denotes the size of the largest stable set in $G$, and $\delta(G)$ denotes the minimum degree of $G$.
For every vertex $v$ of $G$, we denote the (open) neighborhood $\{u \in V(G): uv \in E(G)\}$ of $v$ by $N_G(v)$.
For a subset $S \subseteq V(G)$, we denote the subgraph of $G$ induced by $S$ by $G[S]$.
In this paper, all logarithms are natural, and we define $\logp x=\max\{0,\log x\}$ for $x>0$.

A key ingredient in our proof of \cref{thm:global} is the following result of Delcourt and Postle~\cite{DP25} that reduces the problem of coloring $K_t$-minor-free graphs to coloring graphs on $O(t\log^4 t)$ vertices.

\begin{theorem}[{{\cite[Theorem~1.6]{DP25}}}]
\label{thm:dp}
There is an integer $C_0\ge1$ such that the following holds.
Let $t \geq 3$ be an integer and $G$ a graph. 
Let
$$ f(G,t)=\max_{H \subseteq G}\left\{
 \frac{\chi(H)}a:
 \begin{array}{l}
 \quad a\text{ an integer},\quad t/\sqrt{\log t}\le a\le t,\\
 |V(H)|\le C_0a\log^4a,\quad H\text{ is }K_a\text{-minor-free}
 \end{array}\right\}.$$
If $G$ is $K_t$-minor-free, then
\[
 \chi(G)\le C_0t\bigl(1+f(G,t)\bigr).
\]
\end{theorem}

\cref{thm:dp} allows us to focus on small graphs so that upper bounds on the chromatic number in terms of the number of vertices are effective.
One way to obtain such an upper bound is to repeatedly peel off stable sets of linear size.
Such a stable set exists, due to a result of Duchet and Meyniel~\cite{DM82}.
However, to obtain our stronger bound, we need to exploit stable sets more uniformly rather than one at a time. Because of this, we use a stronger result of Reed and Seymour~\cite{RS98} on fractional chromatic number.

Let $G$ be a graph.
Let $\I(G)$ be the set of all stable sets in $G$.
A {\it fractional coloring} of $G$ is a function $f$ that assigns each stable set in $G$ a nonnegative real number such that $\sum_{I \ni v, I \in \I(G)} f(I) = 1$ for every $v \in V(G)$; the {\it weight} of a fractional coloring is $\sum_{I \in \I(G)} f(I)$.
It is not hard to see that if there exists a fractional coloring of $G$ of weight $r$, then for any function that assigns weights to vertices of $G$, there exists a stable set whose total weight is at least $1/r$ times the total weight on $V(G)$.
This intuitively suggests that a fractional coloring covers the vertices of a graph by using stable sets uniformly.
The {\it fractional chromatic number of $G$}, denoted by $\chi_f(G)$, is the minimum weight of a fractional coloring.

\begin{theorem}[{{\cite[Theorem~(1.3)]{RS98}}}]\label{thm:rs}
For every integer $t\ge2$, if $G$ is a $K_t$-minor-free graph, then $$\chi_f(G)\le2(t-1).$$
\end{theorem}

As mentioned above, the upper bound on the chromatic number obtained by repeatedly peeling off stable sets of size $O(n/t)$ is not sufficient to prove \cref{thm:global}.
To prove \cref{thm:global}, we will use a variant of this idea. 
Instead of peeling off a stable set, one can contract each star in a set of pairwise disjoint stars whose centers are pairwise non-adjacent and the leaves of each star are pairwise non-adjacent into a single vertex; this operation sacrifices one color to construct a minor that has fewer vertices, where the number of vertices reduced is determined by the number of the leaves.
On the other hand, low degree vertices are easy to color, so in fact we only need to try to reduce the number of high degree vertices.
This suggests that we should repeat this contraction process while only focusing on those high degree vertices.

We say that a graph $G$ is {\it $d$-degenerate} if every subgraph of $G$ has a vertex of degree at most $d$.
Note that for every $d$-degenerate graph $G$, there exists a linear ordering $v_1,v_2,...,v_{|V(G)|}$ of $V(G)$ such that for every $1 \leq i \leq |V(G)|$, $v_i$ is adjacent to at most $d$ vertices in $\{v_j: 1 \leq j \leq i\}$.
By using this linear ordering, we can see that every $d$-degenerate graph has chromatic number at most $d+1$.

The first step of the proof is to show that we can sacrifice some colors to construct a minor with small degeneracy.
More precisely, we will prove the following lemma in \cref{sec:pf_rcd}.

\begin{lemma}  \label{lem:rcd_intro}
Let $r\ge2$ and $d \geq 1$ be integers. 
If $G$ is a graph such that $\alpha(H) \geq |V(H)|/r$ for every minor $H$ of $G$, then there exists an integer $s$ with $0 \leq s \leq 1+\frac{r^2+|V(G)|}{d}\log^+(\frac{|V(G)|}{d})$ and there exists a minor $J$ of $G$ such that $J$ is $(d-1)$-degenerate and $\chi(G) \leq \chi(J)+s$.
\end{lemma}

\cref{lem:rcd_intro} allows us to focus on graphs with small degeneracy.
Also, \cref{thm:rs} allows us to assume those graphs have small fractional chromatic number.
Then we use the uniform covering by stable sets provided by a fractional coloring of small weight to construct random stable sets whose deletion substantially reduces the degeneracy, thus yielding a bound on the chromatic number.
This idea will be used to prove the following lemma in \cref{sec:pf_rounding}.

\begin{lemma}\label{lem:rounding_intro}
Let $r \geq 2$ and $d \geq 0$ be integers.
Let $G$ be a $d$-degenerate graph on $n$ vertices with $\chi_f(G) \leq r$.
Let $R=2r$, $m=\lceil n/r\rceil$ and $k=\left\lceil R\log\left(2\max\{1,d/R\}\right)\right\rceil$.
If $R^2 \geq 64km^2\log n$, then $\chi(G)\le R+k$.
\end{lemma}

\cref{lem:rounding_intro} shows that we do not need to color a $d$-degenerate graph naively using $d+1$ colors.  Instead, the fractional
coloring provides stable sets that cover the vertices uniformly. After using about $k=O(r\log(d/r))$ of these stable sets, the remaining graph has degeneracy $R=O(r)$. 

In \cref{sec:coloring_bound} we combine \cref{lem:rcd_intro} and \cref{lem:rounding_intro} to give the following bound on the chromatic number, assuming that all minors have bounded fractional chromatic number.

\begin{lemma}\label{lem:small_intro}
For any real numbers $a,b>0$, there exists an integer $r_0=r_0(a,b)$ such that the following holds.
Let $r\ge r_0$ be an integer, and let $G$ be a graph on $n \leq ar\log^b r$ vertices.
If $\chi_f(H)\le r$ for every minor $H$ of $G$, then $$\chi(G)\le8r\bigl(1+\log(1+\logp(n/r))\bigr).$$
\end{lemma}
Lemma~\ref{lem:small_intro} follows from \cref{lem:rcd_intro,lem:rounding_intro} because \cref{lem:rcd_intro} shows that we can contract down to a $d$-degenerate graph without losing too many colors, and \cref{lem:rounding_intro} shows that we can color the resulting graph without too many more colors.

Finally, we combine \cref{thm:dp} and \cref{lem:small_intro} to prove \cref{thm:global} in \cref{sec:pf_main}.

\section{Reducing the degeneracy} \label{sec:pf_rcd}

\begin{lemma} \label{lem:contract_1}
Let $r\ge2$ and $d \geq 1$ be integers. 
Let $G$ be a graph, and let $A$ be an induced subgraph with $\delta(A) \geq d$.
If $\alpha(H) \geq |V(H)|/r$ for every induced subgraph $H$ of $A$, then there exists a set $F$ of pairwise disjoint connected subgraphs of $A$ such that the graph $J$ obtained from $G$ by contracting each member of $F$ into a vertex satisfies $\chi(G) \leq \chi(J)+1$ and 
$$|V(J)|\le |V(G)|-\frac{d}{r^2+|V(A)|}|V(A)|.$$
\end{lemma}

\begin{proof}
By assumption, there exists a stable set $I=\{v_1,v_2,...,v_{|I|}\}$ in $A$ with $|I| \geq |V(A)|/r$.
Since $|I|$ is an integer, $|I| \geq \lceil |V(A)|/r \rceil$.

Now we define pairwise disjoint stable sets $L_1,L_2,...,L_{|I|}$ such that each $L_i$ is a subset of $N_A(v_i)$ successively.
Let $i$ be an integer with $1 \leq i \leq |I|$ such that $L_1,L_2,...,L_{i-1}$ have been defined.
Define $L_i$ to be a largest stable set in the subgraph of $G$ induced by $N_A(v_i)-\bigcup_{j=1}^{i-1}L_j$.

For every $0 \leq i \leq |I|$, let $S_i = \bigcup_{j=1}^iL_j$ and let $s_i=|S_i|$.
Note that for every $1 \leq i \leq |I|$, we have $$s_i = s_{i-1}+|L_i| \geq s_{i-1}+\frac{|N_A(v_i)|-s_{i-1}}{r} \geq s_{i-1}+\frac{d-s_{i-1}}{r} = (1-\frac{1}{r})s_{i-1}+\frac{d}{r}.$$

\medskip

\noindent{\bf Claim 1:} For every $1 \leq i \leq |I|$, $$s_i \geq d(1-(1-\frac{1}{r})^i).$$

\noindent{\bf Proof of Claim 1:}
We shall prove this claim by induction on $i$.
When $i=1$, $s_i=|L_i| \geq \frac{d-s_0}{r} = \frac{d}{r}$, so the claim holds.
Now we assume that $s_{i-1} \geq d(1-(1-\frac{1}{r})^{i-1})$. 
Then $s_i \geq (1-\frac{1}{r})s_{i-1}+\frac{d}{r} \geq (1-\frac{1}{r})d(1-(1-\frac{1}{r})^{i-1})+\frac{d}{r} = d(1-(1-\frac{1}{r})^i)$.
$\Box$

\medskip

Let $m=|I|$.
By Claim 1, $s_m \geq d(1-(1-\frac{1}{r})^m)$.
Since $r \geq 2$ and $m>0$, we have
$$\left(1-\frac1r\right)^{-m}\ge1+\frac mr,$$
so
$$1-\left(1-\frac1r\right)^m\ge\frac{m}{r+m}.$$
Therefore,
$$s_m\ge\frac{dm}{r+m} \geq \frac{d|V(A)|}{r^2+|V(A)|},$$
where the last inequality follows from $m\ge |V(A)|/r$.

Since the vertices in $I$ are pairwise non-adjacent, and each $L_i$ is a subset of $N_A(v_i)$, we know that each $L_i$ is disjoint from $I$.
Let $F = \{G[\{v_i\} \cup L_i]: 1 \leq i \leq |I|\}$.
Let $J$ be the graph obtained from $G$ by contracting each member of $F$ into a vertex.

For every $1 \leq i \leq |I|$, let $u_i$ be the vertex of $J$ obtained by identifying $\{v_i\} \cup L_i$.
By the definition of the chromatic number, there exists a partition of $V(J)$ into at most $\chi(J)$ stable sets in $J$.
By replacing each $u_i$ by $L_i$, we obtain a partition of $V(G)-I$ into at most $\chi(J)$ stable sets in $G-I$ since each $L_i$ is a stable set.
Hence $\chi(G-I) \leq \chi(J)$.
Since $I$ is a stable set, $\chi(G) \leq \chi(G-I)+1 \leq \chi(J)+1$.

In addition, since $L_1,L_2,...,L_{|I|}$ are pairwise disjoint, 
$$|V(J)|=|V(G)|-\sum_{i=1}^{|I|}|L_i| = |V(G)|-s_m \leq |V(G)|-\frac{d}{r^2+|V(A)|}|V(A)|.$$
\end{proof}

Now we are ready to prove \cref{lem:rcd_intro}.
We restate it for the convenience of the readers.

\begin{lemma}  \label{lem:rcd}
Let $r\ge2$ and $d \geq 1$ be integers. 
If $G$ is a graph such that $\alpha(H) \geq |V(H)|/r$ for every minor $H$ of $G$, then there exists an integer $s$ with $0 \leq s \leq 1+\frac{r^2+|V(G)|}{d}\log^+(\frac{|V(G)|}{d})$ and there exists a minor $J$ of $G$ such that $J$ is $(d-1)$-degenerate and $\chi(G) \leq \chi(J)+s$.
\end{lemma}

\begin{proof}
Let $G_0=G$, and let $A_0$ be a (possibly null) maximal induced subgraph of $G_0$ with minimum degree at least $d$.
For every integer $i \geq 0$, if $A_i$ is defined and non-null, then by \cref{lem:contract_1}, there exists a set $F_i$ of pairwise disjoint connected subgraphs of $A_{i}$ such that the graph $G_{i+1}$ obtained from $G_i$ by contracting each member of $F_i$ into a vertex satisfies $\chi(G_{i}) \leq \chi(G_{i+1})+1$ and $|V(G_{i+1})|\le |V(G_{i})|-\frac{d}{r^2+|V(A_i)|}|V(A_i)|$; let \(A_{i+1}\) be a maximal induced subgraph of $G_{i+1}$ with minimum degree at least $d$, if one exists, and otherwise let $A_{i+1}=\emptyset$.
Let $s$ be the largest integer such that $G_s$ and $A_s$ are defined.

Clearly, $\chi(G) \leq \chi(G_s)+s$, and $G_s$ is a minor of $G$. 
By the definition of $s$, $A_s$ is null, so $G_s$ is $(d-1)$-degenerate.
Let $J=G_s$.
Hence, to prove this lemma, it suffices to bound $s$.

Let $\theta=d/(r^2+|V(G)|)$. 
Because all contractions from $G_i$ to $G_{i+1}$ are within $A_i$, maximality of $A_i$ implies that all vertices of $A_{i+1}$ come from $A_i$.
For every $0 \leq i \leq s-1$, we know that $|V(A_{i+1})|\le |V(A_i)|-\frac{d}{r^2+|V(A_i)|}|V(A_i)| \leq (1-\frac{d}{r^2+|V(G)|})|V(A_i)| = (1-\theta)|V(A_i)|$.
So $|V(A_i)| \leq (1-\theta)^i|V(G)|$ for every $0 \leq i \leq s$.

We may assume $s \geq 1$, for otherwise we are done.
So $A_{s-1}$ is non-empty.
Hence $d+1 \leq |V(A_{s-1})| \leq (1-\theta)^{s-1}|V(G)|$.
Therefore, 
$$s-1
 \leq \frac{\log(|V(G)|/(d+1))}{-\log(1-\theta)}
 \leq \frac{\log(|V(G)|/(d+1))}{\theta}
 \leq \frac{r^2+|V(G)|}{d}\log(|V(G)|/d),$$
where we used $-\log(1-\theta)\ge\theta$.
\end{proof}

\section{Rounding a fractional coloring}  \label{sec:pf_rounding}
We first show that, at the cost of a factor of two, a fractional coloring can be represented by a distribution on stable sets of bounded size.

\begin{lemma}\label{lem:bounded}
Let $r\ge2$ be an integer. 
If $G$ is a graph on $n$ vertices with $\chi_f(G)\le r$, then there exists a multiset $\J$ of stable sets in $G$ and a probability distribution over $\J$ such that for every random stable set $I \in \J$ drawn from this probability distribution, 
    \begin{enumerate}
        \item $\Pr(v\in I)=\frac{1}{2r}$ for every $v\in V(G)$, and
        \item $|I|\le \lceil n/r\rceil$.
    \end{enumerate}
\end{lemma}

\begin{proof}
Let $m=\lceil n/r\rceil$.
Let $f$ be a fractional coloring of $G$ with weight at most $r$.
For every stable set $I$ in $G$, let $P_I$ be a partition of $I$ into $\lceil |I|/m \rceil$ subsets of size at most $m$, and let $g(M)=f(I)$ for each $M \in P_I$.
Let $\J_0$ be the multiset $\bigcup_{I \in \I(G)}P_I$ (recall that $\I(G)$ denotes the set of all stable sets in $G$).
Note that $g$ is a function with domain $\J_0$, and $\sum_{M \ni v, M \in \J_0} g(M) = \sum_{I \ni v, I \in \I(G)} f(I) = 1$ for every $v \in V(G)$.

We remark that $$\sum_{M \in \J_0}g(M) = \sum_{I \in \I(G)}f(I)\left\lceil\frac{|I|}{m}\right\rceil \leq \sum_{I \in \I(G)}f(I) + \frac{1}m\sum_{I \in \I(G)}f(I)|I|.$$
Since the weight of $f$ is at most $r$, we have $\sum_{I \in \I(G)}f(I) \leq r$.
Moreover, $\sum_{I \in \I(G)}f(I)|I| = \sum_{v \in V(G)}\sum_{I \ni v, I \in \I(G)}f(I)=\sum_{v \in V(G)}1=n$.
Hence $\sum_{M \in \J_0}g(M) \leq r+\frac{n}{m} \leq 2r$.

Let $\J$ be the multiset $\J_0 \cup \{\emptyset\}$, and let $g(\emptyset)=2r-\sum_{M \in \J_0}g(M)$.
So $\sum_{M\in\J}g(M)=2r$.
Define the probability distribution over $\J$ such that every member $M$ of $\J$ is drawn with probability $g(M)/(2r)$.
Then this probability distribution satisfies the conclusion of this lemma since for every $v \in V(G)$, $\Pr(v \in I) = \sum_{M \in \J, v \in M}\frac{g(M)}{2r} = \sum_{M \in \J_0, v \in M}\frac{g(M)}{2r}=\frac{1}{2r}$.
\end{proof}

We will use the following form of Azuma's Inequality.

\begin{theorem}[Azuma's Inequality \cite{A67}] \label{azuma}
Let $X$ be a random variable determined by $n$ trials $T_1,T_2,...,T_n$ such that, for each $i$ and any two sequences of outcomes $t_1,t_2,...,t_i$ and $t_1,t_2,...,t_{i-1},t_i'$, 
$$|\E[X|T_1=t_1,T_2=t_2,...,T_i=t_i] - \E[X|T_1=t_1,T_2=t_2,...,T_{i-1}=t_{i-1},T_i=t_i']| \leq c_i.$$
Then $\Pr(|X-E[X]|\geq t) \leq 2e^{-\frac{t^2}{2\sum_{i=1}^nc_i^2}}$.
\end{theorem}

Now we are ready to prove \cref{lem:rounding_intro}.
We restate it for the convenience of the readers.

\begin{lemma}\label{lem:rounding}
Let $r \geq 2$ and $d \geq 0$ be integers.
Let $G$ be a $d$-degenerate graph on $n$ vertices with $\chi_f(G) \leq r$.
Let $R=2r$, $m=\lceil n/r\rceil$ and $k=\left\lceil R\log\left(2\max\{1,d/R\}\right)\right\rceil$.
If $R^2 \geq 64km^2\log n$, then $\chi(G)\le R+k$.
\end{lemma}

\begin{proof}
We may assume that $n\geq3$, since otherwise $\chi(G)\leq n\leq R+k$.
Since $G$ is $d$-degenerate, there exists a linear ordering $v_1,v_2,...,v_{|V(G)|}$ of $V(G)$ such that for every $1 \leq i \leq |V(G)|$, $v_i$ is adjacent to at most $d$ vertices in $\{v_j: 1 \leq j \leq i\}$; for every $1 \leq i \leq |V(G)|$, let $S_i = N_G(v_i) \cap \{v_j: 1 \leq j \leq i\}$.
For every $v \in V(G)$, let $S_v=S_i$, where $i$ is the index such that $v=v_i$.
Note that $|S_v| \leq d$ for every $v \in V(G)$.

Independently choose random stable sets $I_1,\ldots,I_k$ in $G$ according to the distribution given by \cref{lem:bounded}.
For each vertex $v$, let $Z_v=|S_v-(I_1\cup\cdots\cup I_k)|$.
By \cref{lem:bounded},
$$\E Z_v
 \le d\left(1-\frac1R\right)^k
 \le de^{-k/R}
 \le\frac R2.$$
Note that replacing one of the sampled stable sets changes $Z_v$ by at most $2m$.
For every $v \in V(G)$, by \cref{azuma} (with $t=R/2$ and $c_i=2m$ for all $i$), 
$$\Pr(Z_v\ge R) \leq \Pr(|Z_v-\E Z_v|\ge R/2) \leq 2e^{-\frac{R^2/4}{2k \cdot 4m^2}} \leq 2e^{-\frac{R^2}{32km^2}}.$$
Since $R^2 \geq 64km^2\log n$, we know $\Pr(Z_v\ge R) \leq 2/n^2$ for every $v \in V(G)$.
By the union bound, with positive probability, there exists a choice of $I_1,I_2,...,I_k$ such that $|S_v-(I_1\cup \cdots\cup I_k)| < R$ for every $v\in V(G)$. Since this quantity is an integer, $|S_v-(I_1\cup \cdots\cup I_k)| \leq R-1$ for every $v\in V(G)$.
Hence $G-(I_1\cup \cdots\cup I_k)$ is $(R-1)$-degenerate.
So $\chi(G-(I_1\cup \cdots\cup I_k)) \leq R$.
Therefore, $\chi(G) \leq R+k$.
\end{proof}

\section{Bounding the chromatic number} \label{sec:coloring_bound}

\begin{lemma}\label{lem:quantitative}
Let $r \geq 2$ be an integer, and let $G$ be a graph on $n$ vertices with $r \leq n \leq r^2$.
Let $\lambda=1+\log(n/r)$ and $m=\lceil n/r\rceil$.
If $r \ge 48m^2(1+\log\lambda)\log n$ and $\chi_f(H)\le r$ for every minor $H$ of $G$, then $$\chi(G)\le8r(1+\log\lambda).$$
\end{lemma}

\begin{proof}
Let $d=\lceil r\lambda\rceil$.
For every minor $H$ of $G$, since $\chi_f(H)\le r$, we know $\alpha(H)\ge |V(H)|/r$.
By \cref{lem:rcd}, there exist an integer $s$ with $s\le1+\frac{r^2+n}{d}\logp(n/d)$ and a minor $J$ of $G$ such that $J$ is $(d-1)$-degenerate and $\chi(G)\le\chi(J)+s$.

Since $n \leq r^2$, $d\geq r\lambda \geq r$, and $\logp(n/d)\le\log(n/r)\le\lambda$, we have
\begin{equation}\label{eq:rounds}
 s\le1+\frac{2r^2}{r\lambda}\lambda\le2r+1.
\end{equation}

Let $R=2r$ and $$k=\left\lceil R\log\left(2\max\{1,(d-1)/R\}\right)\right\rceil.$$
Since $d-1<r\lambda$ and $\lambda\ge1$,
\begin{equation}\label{eq:rounds2}
k\le2r\log(2\lambda)+1
 \le3r(1+\log\lambda).
\end{equation}
So $$64km^2\log n \leq 192r(1+\log\lambda)m^2\log n \leq 4r^2=R^2,$$
where the last inequality uses the assumption $r\ge 48m^2(1+\log\lambda)\log n$ of this lemma.

Since $J$ is a minor of $G$, $\chi_f(J)\le r$ by assumption of this lemma.
Hence \cref{lem:rounding} implies that $\chi(J)\le R+k$.
Together with \eqref{eq:rounds} and \eqref{eq:rounds2},
$$\chi(G) \leq \chi(J)+s \leq R+k+s \leq 4r+1+3r(1+\log\lambda) \leq8r(1+\log\lambda),$$
as desired.
\end{proof}

Now we are ready to prove \cref{lem:small_intro}.
We restate it for the convenience of the readers.

\begin{lemma}\label{lem:small}
For any real numbers $a,b>0$, there exists an integer $r_0=r_0(a,b)$ such that the following holds.
Let $r\ge r_0$ be an integer, and let $G$ be a graph on $n \leq ar\log^b r$ vertices.
If $\chi_f(H)\le r$ for every minor $H$ of $G$, then $$\chi(G)\le8r\bigl(1+\log(1+\logp(n/r))\bigr).$$
\end{lemma}

\begin{proof}
We may assume $n \geq r$, for otherwise $\chi(G)\leq n<r$ and the lemma holds.
Let $\lambda=1+\log(n/r)$ and $m=\lceil n/r\rceil$.

Note that $m \leq 2a\log^b r$, and $1+\log\lambda \leq 1 + \log(1+\log(a\log^b r)) = 1+ \log(1+\log a +b\log\log r) \leq c_1(1+\log\log\log r)$ for some constant $c_1$ depending only on $a,b$, and $\log n \leq c_2 \log r$ for some constant $c_2$ dependent only on $a,b$.
Hence 
$$48m^2(1+\log\lambda)\log n
 \leq c_3 \left(\log^{2b+1} r
 (1+\log\log\log r)\right) = o(r)$$ for some $c_3$ dependent only on $a,b$.
 
Also $n\leq ar\log^b r \leq r^2$ for all $r \geq c_4$ for some constant $c_4$ depending only on $a,b$.
So there exists $r_0$ dependent only on $a,b$ such that the assumptions of \cref{lem:quantitative} hold for every $r\ge r_0$.
Since $1+\log\lambda=1+\log(1+\log(n/r))$, the conclusion of \cref{lem:quantitative} is exactly the desired bound.
\end{proof}

\section{Proof of \cref{thm:global}} \label{sec:pf_main}

Now we prove \cref{thm:global}.
The following is a restatement.

\begin{theorem}
There is an absolute constant $C$ such that, for every integer $t\geq16$
and every $K_t$-minor-free graph $G$,
$$\chi(G)\le Ct\log\log\log t.$$
\end{theorem}

\begin{proof}
Let $C_0$ be the constant in \cref{thm:dp}.
It suffices to prove the theorem for all sufficiently large $t$, since the finitely many remaining values can be absorbed by increasing $C$.
Let $G$ be a $K_t$-minor-free graph, and let $f(G,t)$ be as defined in \cref{thm:dp}.

Let $a$ be an integer such that
$t/\sqrt{\log t}\le a\le t$, and let $H$ be a $K_a$-minor-free subgraph of $G$ on at most $C_0a\log^4a$ vertices.
Let $r=2(a-1)$.
Every minor $J$ of $H$ is $K_a$-minor-free, so \cref{thm:rs} implies $\chi_f(J)\le 2(a-1)=r$.
Moreover, for all $a \geq 3$, we have $|V(H)|\le C_0a\log^4a\le C_0r\log^4r$.

Let $r_0$ be the corresponding number in \cref{lem:small} (taking $a=C_0$ and $b=4$).
Since $a\ge t/\sqrt{\log t}$, we may assume $a \geq r_0$, for otherwise $t$ is upper bounded by an absolute constant and the theorem holds.
Hence $r \geq a \geq r_0$.
Then, since $r=2(a-1)$ and $|V(H)|\le C_0a\log^4a$, \cref{lem:small} implies that there is an absolute constant $C_1$ depending only on $C_0$ such that
$$
\chi(H) \leq 8r\bigl(1+\log(1+\logp(|V(H)|/r))\bigr) \leq C_1a\log\log\log a.
$$
Since $a\le t$,
$$
\frac{\chi(H)}a
\le C_1\log\log\log a
\le C_1\log\log\log t.
$$
Therefore,
$$ 
f(G,t)\le C_1\log\log\log t.
$$
By \cref{thm:dp}, $$ \chi(G)
 \le C_0t(1+f(G,t))
 \le Ct\log\log\log t$$
for a suitable absolute constant $C$.
This proves the theorem.
\end{proof}

\paragraph{Declaration of AI use.} ChatGPT Codex was used to help develop the proof. The authors supplied a high-level plan of attack and guided the development of the core argument. The authors independently verified and rewrote all AI-generated material in their own words. The authors take full responsibility for all mathematical claims in the paper.

\end{document}